\documentclass[12pt]{amsart}
\usepackage{amsmath,amsthm,amsfonts,amscd,amssymb,eucal,latexsym, fullpage}
\usepackage{mathrsfs}
\usepackage[numbers]{natbib}
\usepackage[fit]{truncate}
\usepackage{fullpage}
\usepackage{graphicx}
\usepackage{epsfig}
\usepackage{color,soul}
\usepackage[all,cmtip]{xy}
\usepackage[top=1.5in, bottom=1.5in, left=1.5in, right=1.5in]{geometry}

\theoremstyle{plain}
\newtheorem{theorem}{Theorem}[section]

\newtheorem{lemma}[theorem]{Lemma}

\newtheorem{question}{Question}

\theoremstyle{definition}
\newtheorem{definition}[theorem]{Definition}
\newtheorem{remark}[theorem]{Remark}

\def\de{\delta}    \def\Sg{\Sigma} \def\sg{\sigma}
\def\su{\subset}     \def\Lm{\Lambda}
\def\al{\alpha}       \def\ga{\gamma} \def\Ga{\Gamma}
 \def\ra{\rightarrow}\def\om{\omega} \def\OM{\Omega}
\def\ti{\tilde}     \def\z{\times}  \def\x{\cdot}
   \def\ep{\varepsilon}    

\newcommand{\R}{\mathbb{R}}

\newcommand{\Mm}{\mathcal{M}}
\newcommand{\Bb}{\mathcal{B}}
\newcommand{\Nn}{\mathcal{N}}

\newcommand{\Ll}{\mathcal{L}}
\newcommand{\Tt}{\mathcal{T}}
\newcommand{\Oo}{\mathcal{O}}
\newcommand{\Vv}{\mathcal{V}}
\newcommand{\Ee}{\mathcal{E}}

\newcommand{\mb}{\mathbf}

\DeclareMathOperator{\diam}{diam}

\DeclareMathOperator{\length}{length}

\DeclareMathOperator{\vol}{vol}

\begin{document}

\title{Length of closed geodesics on Riemannian manifolds with good covers}
\author{Zhifei Zhu}
\maketitle
\begin{abstract}
In this article, we prove a generalization of our previous result in \cite{r_zzf}. In particular, we 
show that for an $n$-dimensional, simply-connected Riemannian manifold with diameter $D$ and volume $V$. Suppose that $M$ admits a good cover consisting of $N$ elements. Then the length of a shortest closed geodesic on $M$ is bounded by some function which only depends on $V, D$ and $N$.
\end{abstract}


\section{Introduction}

In this paper, we extend the results established in \cite{r_zzf}. Specifically, we prove the following:

\begin{theorem}\label{thm_main}
Let $M$ be a complete, $n$-dimensional, simply-connected Riemannian manifold with diameter $D$ and volume $V$. Suppose that $M$ admits a good cover consisting of $N$ elements. Then the length of the shortest closed geodesic on $M$ satisfies:
\[
\length \leq A(n, V, D, N),
\]
where $A$ is a function that only depends on $n$, $V$, $D$ and $N$.
\end{theorem}

The concept of a good cover used in Theorem~\ref{thm_main} is defined as follows:

\begin{definition}\label{def_gc}
Let $M$ be a closed, $n$-dimensional, simply-connected Riemannian manifold with diameter $D$ and volume $V$. An open cover $\Oo = \{B_i\}_{i \in I}$ of $M$ is said to be a good cover if there exists a refinement $\ti{\Oo} = \{A_j\}_{j \in J}$ that satisfies the following conditions:
\begin{itemize}
  \item Any two points in $A_i$ can be joined by a curve of length $\leq F_A(V, D)$ within $A_i$.
  \item Any two points in $B_j$ can be joined by a curve of length $\leq F_B(V, D)$ within $B_j$.
  \item If the intersection $A_i \cap A_j \neq \emptyset$, then the union $A_i \cup A_j$ is contained in some $B_k \in \Oo$.
  \item For any closed curve $\gamma: S^1 \to B_j \in \Oo$, there is a contraction $H: S^1 \times [0,1] \to B_j$ of $\gamma$, i.e., $H(t,0) = \gamma(t)$ and $H(t,1) = \text{point} \in B_j$, such that
  \[
  w_H \leq G(V, D).
  \]
  \item The number of elements in $\ti{\Oo}$ and $\Oo$ are both bounded by some $N$.
\end{itemize}
\end{definition}

\begin{remark}\label{rm_fup}
Sometimes, we take $F(V, D) = \max\{F_A(V, D), F_B(V, D)\}$ as a universal upper bound for an arc connecting points within these open sets. If, in particular, $N=N(V,D)$, then, obviously, the function $A=A(V,D)$ in Theorem~\ref{thm_main}.
\end{remark}

In Section~\ref{sec21}, we provide some examples of manifolds that admit a good cover, especially the cases of four-dimensional non-collapsed manifolds with bounded diameter and Ricci curvature considered in \cite{r_zzf}. In particular, in Definition~\ref{def_gc}, we do not have restrictions on the lower bound of diameters of the sets $A_i$ or $B_j$. In other words, it is possible that for some $A_i$ or $B_j$, their diameters are arbitrarily close to zero. 

While the elements $A_i$ or $B_j$ may not be geodesically convex, it is helpful to consider these elements as the $C^{1,\alpha}$-images of some convex domains or geodesic balls. In fact, if they are constructed to be geodesically convex, then one can take the obvious bound $F = D$ for arcs connecting two points in these sets. 

Intuitively, the refinement ``$A_i$'' describes how the elements ``$B_i$'' intersect each other. Therefore, one may ask about the relation between the number of diffeomorphic types in some collection of manifolds 
$$\Mm = \{M \text{ admits a good cover}\},$$
and the given parameters $V$, $D$ and $N$. Currently, we do not have an affirmative answer for this relation, which can be summarized in the following question.

\begin{question}
Is it true that if a collection of simply-connected Riemannian manifolds with diameter $D$ and volume $V$ only consists of $N$ diffeomorphic types, then the length of a shortest closed geodesic can be bounded by some function of $D$, $V$, and $N$?
\end{question}

The first study of the length of a shortest closed geodesic on Riemannian manifolds (which are homeomorphic to $T^2$ or $\mathbb{R}P^2$) is due to C.~Leowner and P.~Pu, appears in the work \cite{pu1} of Pu. For simply-connected manifolds, C. Croke \cite{croke1988area} obtained the first upperbound for Riemannian $S^2$. M. Gromov asked in \cite{gromov1983filling} whether the length of a shortest periodic geodesic in a $n$-dimensional Riemannian manifold $M^n$ can be bounded by $c(n)\vol (M)^{1/n}$. Similar questions can be asked for the diameter $D$ of the manifold. There are various theories developed to study this question, one may refer to \cite{katz} for an overview of this topic.

It is a theorem of Lusternik and Fet that on every compact simply-connected manifold, there exists at least one closed geodesic. (See, for example, \cite{klingenberg2012lectures}, or \cite{milnor}). The outline of the approach for obtaining an upper bound for the length of this closed geodesic is the following. Let us consider the loop space $\OM_p M$, which consists of all loops based at a fixed point \( p \in M \). Suppose the smallest integer \( m \) satisfies \(\pi_{m+1}(M) \neq 0\). If one can construct a "small" non-contractible \( m \)-dimensional sphere in \(\OM_p M\), that is, a non-contractible map \( S^m \to \OM_p^L M \), where \(\OM_p^L M\) denotes the subspace of \(\OM_p M\) consisting of loops of length \(\leq L\), then, by a Morse-type argument, there exists a closed geodesic of length \(\leq L\). This geodesic arises as a critical point of the length functional on the free loop space \(\Lm M\). 

A.~Nabutovsky and R.~Rotman demonstrated in \cite{nabutovsky2013length} that the primary obstruction to constructing these "small" non-contractible spheres is the existence of certain "short" closed geodesics on \( M \) (see \cite[Corollary 5.4]{nabutovsky2013length}). To formalize these ideas, they introduced the notions of the depth of a loop ( \cite[Definition~7.1 \& 7.4]{nabutovsky2013length}) and the width of a homotopy.

\begin{definition}[Depth of a loop]\label{def3_d}
Let $M$ be a closed $n$-dimensional simply-connected Riemannian manifold with diameter $D$, and let $\gamma: S^1 \to M$ be a loop in $M$ based at $p$. We define the depth $S(\gamma)$ of $\gamma$ as the infimum of positive numbers $S$ such that $\gamma$ is contractible by a path homotopy through loops of length $\leq \length(\gamma) + S$. 

We also define $S_p(M, L)$ as 
\[
S_p(M, L) = \sup_{\length(\gamma) \leq L} S(\gamma),
\]
where the supremum is taken over all loops $\gamma$ of length $\leq L$ based at $p$. 
\end{definition}

\begin{definition}[Width of a homotopy]\label{def1_w}
Let $M$ be a Riemannian manifold, and let $\gamma_i: [0,1] \to M$, $i=1,2$, be two curves in $M$. Suppose $\gamma_1$ and $\gamma_2$ are homotopic, and let $H: [0,1] \times [0,1] \to M$ be a homotopy between $\gamma_1$ and $\gamma_2$. For each fixed $s \in [0,1]$, denote 
\[
H_s = H(s, \cdot): [0,1] \to M,
\]
the curve describing the trajectory of a point $H(s,0)$ during the homotopy. We define the width $\omega_H$ of the homotopy $H$ as
\[
w_H = \max_{s \in [0,1]} \text{length of } H_s.
\]
\end{definition}
And in particular, 
\begin{theorem}\label{thmD}
If for any closed curve $\gamma$ of length bounded by $L$, there exists a contraction of $\gamma$ with width bounded by some constant $W$, then
\[
S_p(M, L) \leq \max \{2L, 2W + 2D\}.
\]
\end{theorem}
The proof of this inequality can be found in \cite[Section~8]{nabutovsky2013length}. In \cite{nabutovsky2013length}, by taking the base point $p = q = x$ in Theorem~7.3 and applying Corollary~5.4, Nabutovsky and Rotman proved the following result:

\begin{theorem}\label{thm4_q}
Let $M^n$ be a closed Riemannian manifold of diameter $D$, and let $p$ be a point in $M$. Suppose $S \geq 0$, and there exists $k \in \mathbb{N}$ such that there is no geodesic loop of length in $((2k-1)D, 2kD]$ based at $p$, which is a local minimum of the length functional on $\OM_p M$ of depth $> S$. Then for every positive integer $m$, any map $f: S^m \to \OM_p M$ is homotopic to a map $\tilde{f}: S^m \to \OM_p^{L + o(1)} M$, where
\[
L = ((4k+2)m + (2k-3))D + (2m-1)S.
\]

In this case, the length of the shortest closed geodesic on $M$ does not exceed 
\[
L = ((4k+2)m + (2k-3))D + (2m-1)S.
\]
\end{theorem}

Therefore, it is enough to show that
\begin{theorem}\label{thmW}
  Let $M$ be a complete, $n$-dimensional, simply-connected Riemannian manifold with diameter $D$ and volume $V$ that admits a good cover consisting of $N$ elements. Suppose that there is no closed geodesic on $M$ with length bounded by $3D$, then the for every closed curve $\ga:S^1\rightarrow M$, there is a homotopy $H$ that contracts $\ga$ with
\[
w_H \leq B(V, D, N),
\]
where $B$ is a function that only depends on $V$, $D$ and $N$.  
\end{theorem}

\section{Nerve of the cover and simplicial approximation}

If a manifold admits a good cover, we may then approximate every contraction of the closed curves in the manifold with homotopies through some particular loops, i.e., cycles
in the 1-skeleton of the nerve of the covering $\ti{\Oo}$. More precisely, let us consider the following.
\begin{definition}\label{def_graph}
  Let $M$ be a closed Riemannian manifold and $\ti{\Oo}=\{A_i\}_{i=1}^{N}$ the refinement of a good covering
  as in Definition~\ref{def_gc}.
  We associate to $(M,\ti{\Oo})$ a graph $\Sg$, which is essentially the nerve of the
  covering $\ti{\Oo}$, constructed in the following way.
   \begin{enumerate}
     \item For each element $A_i\in \ti{\Oo}$, we pick a point $x_i\in A_i$ as a vertex
      in $\Sg$.
     \item If for some $1\leq i<j \leq N$, the open sets $A_i\cap A_j \neq \emptyset$,
     then we connect
     the corresponding vertices $x_i$ and $x_j$ with curve $\ga_{ij}$ of length bounded by $2F_A(V,D)$ whose image is contained in $A_i\cup A_j$.
   \end{enumerate}
\end{definition}

In fact, if $A_i\cap A_j \neq \emptyset$, one may pick a point $z$ in $A_i\cap A_j$ and then the vertices $x_i$ and $x_j$ can be connected by arcs from $x_i$ to $z$ and from $z$ to $x_j$.

\begin{remark}
  Note that the number of the vertices in $\Sg$ equals to N,
  which is the number of the elements
  in $\Oo$. And the number of distinct edges in $\Sg$ is at most $\binom{2}{N}$.
\end{remark}

\begin{definition}\label{def_sp}
 A simplicial curve $\al$ in $\Sg$ is a simplicial map
 $\al:[0,1]_{\triangle}\ra \Sg$,
 where $[0,1]_{\triangle}$ is constructed by partitioning
 the interval $[0,1]$ into $L$ sub-intervals with $0 = t_0 < \dots < t_L =
 1$, and $L\geq 1$ is an integer. We define the simplicial length $m(\al)$
 of $\al$ to be the number of edges in $\al$. In other words, $m(\al)=L$.
 When $\alpha(0) = \alpha(1)$, we refer to $\alpha$ as a loop within $\Sigma$.

 Note that since $\Sg$ consists of edges being curves in $M$, one may
 always view $\alpha:[0,1]\ra \Sg$ as a map $\al:[0,1]\ra M$ through the natural
 embedding $\Sg\hookrightarrow M$. Through out this section, we assume that our manifold $(M,g)$ admits a good cover in the sense of Definition~\ref{def_gc} with diameter $D$ and volume $V$.
\end{definition}

Given a smooth curve $\ga:[0,1]\ra M$, we first show that there
exists a simplicial curve $\al$ in $\Sg\hookrightarrow M$ which is homotopic to $\ga$ though
a homotopy with bounded width. The curve $\al$ is called a simplicial approximation
of the curve $\ga$.  The construction in Lemma~\ref{lm_simp} were originally used in the work \cite{rotman2000upper} of R.~Rotman.

\begin{lemma}\label{lm_simp}
For any curve $\ga:[0,1]\ra M$, there exists a simplicial curve $\al:[0,1]\ra \Sg$
such that $\ga$ is homotopic to $\al$ through a homotopy $H$ of width $w_H\leq G_\Sg(V,D)$.
\end{lemma}

\begin{proof}
Let us denote by $\Oo=\{A_i\}$ the refined cover of the manifold as before.
Since the image of $\ga$ is path-connected, we may choose a sufficiently fine subdivision
of $[0,1]$, say, $0=t_0<t_1<\dots<t_T=1$, such that if $\ga(t_i)\in A_k$
and $\ga(t_{i+1})\in A_l$, then the intersection
$A_k\cap A_l$ is nonempty,
and the segment $\ga([t_i,t_{i+1}])\su A_k\cup A_l$.

Now we may construct the simplicial approximation in the following way. 
Let us connect $\gamma(t_i)$ with $x_k$ by an arc $\sg_{ik}$ of length $\leq F_A$, where $x_k$ is the vertex in $\Sg$ corresponding to $A_k$. 
(and connect $\ga(t_{i+1})$, $x_l$ by $\sg_{i+1,k}$, resp.) Then $ \ga([t_k,t_{k+1}])\cup -\sg_{i+1,k} \cup  -\ga_{kl}\cup \sg_{ik}$ forms
a loop in $A_k\cup A_l$. By the assumption of the good cover, there is a contraction of this loop in some $B_s$ with width $\leq G$, 
which, in particular, induces a homotopy between $\ga([t_i,t_{i+1}])$ and $\sg_{i+1,k} \cup  \ga_{kl}\cup -\sg_{ik}$ with width $\leq G+2F_A$.

By applying this to all the intervals in the subdivision, and observe that for 
consecutive intervals, the arc $\sg_{ik}\cup-\sg_{ik}$ can be contracted to $x_k$ through a homotopy with width $\leq 2F_A$. 
Therefore, we conclude that $\gamma$ is homotopic to $\al:= \cup \gamma_{ik}$ though a homotopy with width $\leq G_\Sg(V,D):=G+4F_A$.
\end{proof}

We will present several essential technical lemmas, apart from the above Simplicial Approximation Lemma~\ref{lm_simp}, that will be used in
the proof of our main results. It's worth noting that these results aligns with those outlined in our paper~\cite{r_zzf}. While we are not providing detailed proofs here, we will present these statements for the sake of completeness.

\begin{lemma}\label{approx minimizing geodesic}
 Let $\ga:[0,1]\ra M$ be a minimizing geosdesic, and $\al$ its approximation as in Lemma~\ref{lm_simp}. Then the simplicial length of $\al$ is bounded by $Z(V,D)$.
\end{lemma}

\begin{lemma}\label{lm_brk}
 Let $\ga:[0,1]\ra M$ be a closed curve, and $\al$ its approximation as in Lemma~\ref{lm_simp}. Then there are explicit functions $X(V,D)$, $W(V,D)$ and $\de(M)$, such that if the simplicial length $m(\al)>X(V,D)$, then $\ga$ is homotopic to $\ga_1\cup \ga_2\cup \dots \cup \ga_k$ through a homopoty $H:[0,1]\z[0,1]\ra M$ which satisfies the following.
 \begin{enumerate}
   \item The width of the homotopy $w_H\leq W(V,D)$.
   \item If $\al_i$ is the simplicial approximation of each $\ga_i$, then $m(\al_i)\leq X(V,D)$.
   \item The length of each $\ga_i\leq \length(\ga)-\de$. 
 \end{enumerate}
\end{lemma}

\begin{lemma}\label{lm_wsep}
Let $\ga:[0,1]\ra M$ be a closed curve. Suppose that $\ga=\cup_{i=1}^n \ga_i$, where each $\ga_i$ is a closed curve with base point $p$. If each $\ga_i$ can be contracted to a point in $M$ through a homotopy with width bounded by $W_i$, then there exists a homotopy $H:[0,1]\z [0,1]\ra M$ such that $H(t,0)=\ga(t)$ and $H(t,1)=p$. And the width $\om_H$ of this homotopy is bounded by $2\cdot \max_i W_i$.
\end{lemma}

\begin{remark}
    The statement of the above lemmas does not depend on the number of the elements $N$ in the good cover.
    \end{remark}
Our main theorem states that if a Riemannian manifold admits a good cover in the above sense, then the length of a shortest closed geodesic, can be bounded by some function depending on $F,G$ and $N$. In fact, there are several classes of manifolds that admit such type of cover.

\subsection{Manifolds with injectivity radius lower bound}
If we consider manifolds with injectivity radius bounded from below by some constant $r_0$, for example, manifolds with sectional curvature $\leq k$, $\vol \leq V$, then in this case, we may construct a good cover with $F=r_0$ and $G=G(r_0(V,D,k))$. In this case, it is showed by R.~Rotman in \cite{rotman2000upper}, that if $M$ has a non-trivial second homology group, then the length of a shortest closed geodesic is bounded above. Our Theorem~\ref{thm_main} does not improve this result. It is worth mentioning that the assumption on the second homology group can be removed by A.~Nabutovsky and R.~Rotman in \cite{nabutovsky2013length}.

\subsection{Non-collapsed manifolds with Ricci curvature bounds}\label{sec21}

In \cite{r_zzf}, we have constructed a good cover for four-dimensional manifolds whose diameter $\leq D$, volume $\geq v$ and $|Ric|<3$. Let us briefly describe the process here. In \cite{cheeger2014regularity}, Cheeger and Naber showed that for any non-collapsed manifold of dimension 4, with bounded
Ricci curvature, there is a structural theorem called bubble-tree decomposition.

\begin{theorem}[\cite{cheeger2014regularity}, Theorem 8.64]\label{thm3_f}
Let $M$ be a 4-dimensional Riemannian manifold with $|Ric|\leq 3$, $\vol(M)>v>0$ and $\diam(M)\leq D$. Then $M$ admits a decomposition into bodies and necks
$$M=\Bb^1 \cup \bigcup_{j_2=1}^{N_2} \Nn_{j_2}^2\cup \bigcup_{j_2=1}^{N_2} \Bb_{j_2}^2\cup\dots\cup \bigcup_{j_k=1}^{N_k} \Nn_{j_k}^k\cup \bigcup_{j_k=1}^{N_k} \Bb_{j_k}^k,$$

such that the following conditions are satisfied:
\begin{enumerate}
\item If $x\in \Bb_i^j$, then $r_h(x)\geq r_0(v,D)\cdot \diam (\Bb_i^j)$, where $r_h$ is the harmonic radius and $r_0$ is a constant that only depends on $v$ and $D$.
\item Each $\Nn_i^j$ is diffeomorphic to $\R\z S^3/\Ga_i^j$ for some $\Ga_i^j\su O(4)$ with the order $|\Ga_i^j|<c(v,D)$.
\item $\Nn_i^j\cap \Bb_i^j$ is diffeomorphic to $\R\z S^3/\Ga_i^j$.
\item $\Nn_i^j\cap \Bb_{i'}^{j-1}$ is either empty or diffeomorphic to $\R\z S^3/\Ga_i^j$.
\item Each $N_i\leq n(v,D)$ and $k\leq k(v,D)$.
\end{enumerate}
\end{theorem}

It is constructed in \cite{r_zzf} that, in this case, the manifold $M$ admits a good cover in the sense of Definition~\ref{def_gc}. Namely,
\begin{lemma}\label{lm_gc4}
Let $M$ be a 4-dimensional Riemannian manifold which satisfies the conditions of Theorem~\ref{thm3_f}.
Then there is an open cover $\Oo$ of the manifold and its refinement $\ti{\Oo}$ in the sense of Definition~\ref{def_gc} and Remark~\ref{rm_fup} with $F=3D$ and $G=21D$. In other words,
\begin{itemize}
  \item Any $x,y$ in $A\in \ti{\Oo}$ can be connected by a curve with length less than $3D$.
  \item For any closed curve $\ga:[0,1]\ra B\in \Oo$ with $\ga(0)=\ga(1)=p$, there is a contraction $H:[0,1]\z [0,1]\ra B$
  with $H(t,0)=\ga(t)$ and $H(t,1)=p$ such that $w_H\leq 21D$.
\end{itemize}
Moreover, the total number of elements in $\ti{\Oo}$ is bounded by some $N(V,D)$ which can be computed by the constants in Theorem~\ref{thm3_f}.
\end{lemma}

The construction of this cover can be found in Section 2 of \cite{r_zzf}. Here we omit the detailed proof but instead we describe the idea. By Theorem~\ref{thm3_f}, $M$ admits a cover with the body regions $\Bb_i^j$ and the neck regions $\Nn_i^j$.

In each body $\Bb_i^j$, for any $x\in \Bb_i^j$, the harmonic radius $r_h(x)\geq r_0\cdot \diam(\Bb_i^j)$ is in proportional to the diameter of the body. Let $R(x)=\frac{r_h(x)}{8\x (1+\ep)}$. It is show in \cite[Lemma 2.2]{} that every closed curve in $B_R(x)$ can be contracted within $B_R(x)$ through a homotopy of width $\leq D$. We may pick a cover $\cup_k B_{R(x_k)/4}(x_k) \supseteq \Bb_i^j$ such that $B_{R(x_k)/16}(x_k)$ are pairwise disjoint. This guarantees the number of the elements in this collection is bounded by some $N_{\Bb_i^j}(V,D)$. Note that the collection $\{B_{R(x_k)}\}$ still form a cover of $\Bb_i^j$. Therefore, we simply choose $B_{R(x_k)/4}(x_k)$ to be an element in $\ti{\Oo}$ and $B_{R(x_k)}(x_k)$ an element in $\Oo$.

The neck regions in general do not admit a lower bound on harmonic radius. However, the geometry of $\Nn_i^j$ is completely understood. In this case, we cover the neck by the image $\{\ti{T}_{i,k}^j\}$ and $\{T_{i,k}^j\}$ of finitely many convex regions $\{\ti{K}_{i,k}^j\}$, and their refinements $\{K_{i,k}^j\}$, respectively.  Here $k \leq N_{\Nn_i^j}(V,D)$; $\ti{K}_{i,k}^j\su \R\z S^3/\Ga_i^j$ is convex; and $N_{\Nn_i^j}(V,D)$ only depends on the order of $|\Ga_i^j|<c(V,D)$. It is then showed in \cite[Lemma 2.8]{} that the cover satisfies Definition~\ref{def_gc} with $F=3D$ and $G=21D$.

\begin{remark}
If the neck region is empty, e.g., the manifold admits an injectivity radius lower bound which only depends on $V$ and $D$, then one can pick the cover in the same way as the body region, with $G=D$ or finer estimates depending on the injectivity radius.
\end{remark}

\section{Proof of the main Theorems}
In this section, we will prove our main Theorem~\ref{thm_main} and Theorem~\ref{thmW}. The idea of the proof is based on our paper \cite{r_zzf}. We first proof Theorem~\ref{thmW}.

\begin{proof}[Proof of Theorem~\ref{thmW}]
 Let $\ga:[0,1]\ra M$ be a closed curve with $\ga(0)=\ga(1)=p$ in $M$. By Lemma~\ref{lm_simp}, there is a loop $\al:[0,1]_{\triangle}\ra \Sg$ such that $\ga$ is homotopic to $\al$ through a homotopy $H_1$ with $w_{H_1}\leq G_\Sg(V,D)$.

 Without lost of generality, we may assume that the simplicial length $m(\al)$ does not exceed $X(V,D)$. Otherwise, by Lemma~\ref{lm_brk}, the curve $\ga$ is homotopic to $\ga_1\cup \dots \cup \ga_k$, such that if $\al_i$ is the approximation of $\ga_i$, then $m(\al_i)\leq X(V,D)$. Furthermore, the length of each $\length(\ga_i)\leq \length(\ga)-\de(V,D)$, for some $\de>0$. Then, by Lemma~\ref{lm_wsep}, if each $\ga_i$ can be contracted through a homotopy of width $\leq W$, then $\ga$ can be contracted through a homotopy of width $\leq 2W$. Therefore, it is suffices to contract each $\ga_i$.

 Below we will show that $\ga$ can be contracted to a point through a homotopy of width bounded by some $B(V,D,N)$.
\begin{enumerate}
  \item[(0)] Birkhoff curve shortening process for free loops (BPFL).

  Note that because the length of $\ga\leq 3D$, under the assumption that there is no closed geodesic whose length is $\leq 3D$, if we apply the Birkhoff curve shortening process for free loops (BPFL) to $\ga$, then there is a length non-increasing contraction of $\ga$.

  Let us denote this contraction by $H:[0,1]\z[0,1]\ra M$, where $H(t,0)=\ga(t)$ and $H(t,1)=$ point. We consider a partition of the second interval $[0,1]=s_0<s_1<\dots<s_l=1$ such that $w_H(s_j,s_{j+1})\leq D$. And let $\ga^{j}(t):=H(t,s_j)$, for every $0\leq j\leq l$.

  \item[(1)] First step.

   By Lemma~\ref{lm_simp}, each $\ga^{j}$ is homotopic to some simplicial $\al^j\su \Sg$. In particular, by assumption, for the first curve $\ga^0:=\ga$, we have $m(\al^0)\leq X(V,D)$.

   Let $\ga^{j_0}$ be the first $j_0\geq 1$ such that $m(\al^{j_0})> X(V,D)$.
   Then $\ga$ is homotopic to $\ga^{j_0}$ through curves $\ga^1, \ga^2,\dots, \ga^{j_0-1}$. The width of homotopy bewteen each $\ga^j$ and $\ga^{j+1}$ is less than or equal to $D$, and the simplicial approximation of each $\al^j$ has simplicial length $m(\al^j)\leq X(V,D)$, for $j=0,1,\dots, j_0-1$.

    We now apply Lemma~\ref{lm_brk} to $\ga^{j_0}$. In this case, $\ga^{j_0}$ is homotopic to some $\ga^{j_0}_{1}\cup \dots \cup \ga^{j_0}_{k}$. Let $\al^{j_0}_{i}$ be the simplicial approximation of $\ga^{j_0}_{i}$. By Lemma~\ref{lm_brk}, the length of each $\ga^{j_0}_{i}$ satisfies $$\length (\ga^{j_0}_{i})\leq \length (\ga) - \de.$$
    And the simplicial length $m(\al^{j_0}_{i})\leq X(V,D)$.

  \item[(2)] A generic step.

  Now we repeat this construction to the curves $\ga^{j_0}_{i}$, which will generate a finite family of curves. (We describe in (3) the parametrization of this family.) Denote by $\mb{i}_l$ the indices $i_1,i_2,\dots, i_l$. During this process, a generic curve
$\ga_{\mb{i}_l}:=\ga_{\mb{i}_l}^0$ is homotopic to some $\ga_{\mb{i}_l}^{j_0}$ through curves $\ga_{\mb{i}_l}^1, \dots,\ga_{\mb{i}_l}^{j_0-1}$ constructed using BPFL, such that if $\al_{\mb{i}_l}^j$ is the approximation of $\ga_{\mb{i}_l}^j$ as in Lemma~\ref{lm_simp}, then $m(\al_{\mb{i}_l}^j)\leq X(V,D)$, for $j=0,1,\dots,j_0-1$ and $m(\al_{\mb{i}_l}^{j_0})> X(V,D)$. The width of homotopy between $\ga_{\mb{i}_l}^j$ and $\ga_{\mb{i}_l}^{j+1}$ is bounded by $D$.

And by Lemma~\ref{lm_brk}, the curve $\ga_{\mb{i}_l}^{j_0}$ is homotopic to $\ga_{\mb{i}_l,1}\cup\dots \cup \ga_{\mb{i}_l,k}$ with the simplicial length of the approximation $\al_{\mb{i}_l,j}$ of each $\ga_{\mb{i}_l,j}$ is bounded by $X(V,D)$.

This process will terminate if no such $j_0$ exists, i.e., $\ga_{\mb{i}_l}$ can be contracted to a point through a family of curves $\ga_{\mb{i}_l}^1, \dots,\ga_{\mb{i}_l}^{j_0-1}=\text{point}$, with simplicial length of the approximation $\leq X(V,D)$.

On the other hand, each time we apply Lemma~\ref{lm_brk},
\begin{equation*}
  \length(\ga_{\mb{i}_l,j})\leq \length(\ga_{\mb{i}_l})-\de \leq \length(\ga)-(l+1)\cdot\de
\end{equation*}

  Note that the above inequality also implies that
  $$l\leq \frac{\length(\ga)}{\de}-1,$$
  where the right-hand side is finite (but may not be bounded by any functions of $V$ and $D$). Hence for each ``branch'' of the curve $\ga_i$, we may apply Lemma~\ref{lm_brk} for at most $\frac{\length(\ga)}{\de}$ times before we get a point during the contraction.

  \item[(3)] Parametrization of the family.

  In previous steps, we generate a finite family of curves denoted as $\{\ga_I\}_{I\in \Vv(\Tt)}$. This family is parameterized by a finite tree $\Tt=\{v_i,e_{ij}\}$, where $\Vv(\Tt)=\{v_i\}$ is the set of vertices and $\Ee(\Tt)=\{e_{ij}\}$ the set of edges. We denote by $h(\Tt)$ the height of a finite tree $\Tt$. The tree $\Tt$ is constructed as follows.
\begin{enumerate}
  \item The root $v_0$ of the tree is defined to be the curve $\ga_i$, so that the simplicial length of the approximation $m(\al_i)\leq X(V,D)$.
  \item Each curve $\ga_{\mb{i}_l}^j$, including the point curves, occurred during the generic step (2) whose approximation has simplicial length bounded by $X(V,D)$ is identified with a vertex $v_t\in \Vv$.

      Since the total number of the curve is finite, we can choose the indices such that if $\ga_{\mb{i}_l}^{j}$ is a subsequent curve of $\ga_{\mb{i}_l}^{k}$ during the BPFL, then the corresponding vertices $v_{t_j}$ and $v_{t_k}$ satisfy $t_j>t_k$.

      In addition, suppose that a curve $\ga_{\mb{i}_l}^{j_0}$ is homotopic to $\ga_{\mb{i}_l,1}\cup\dots \cup \ga_{\mb{i}_l,k}$. Note that the curve $\ga_{\mb{i}_l}^{j_0}$ is not identified with a vertex. In this case, we require that if the curves $\ga_{\mb{i}_l,i}$, where $1\leq i\leq k$, and $\ga_{\mb{i}_l}^{j_0-1}$ are identified with the vertices $v_{s_1},\dots,v_{s_k}$ and $v_{t_{j_0-1}}$, then $s_i>t_{j_0-1}$.
  \item Suppose that $\ga_{\mb{i}_l}^{j+1}$ is a subsequent curve of $\ga_{\mb{i}_l}^{j}$ during BPFL process and they are identified with vertices $v_{t_j}$, $v_{t_{j+1}}$, respectively. We define an edge $e_{t_j,t_{j+1}}\in \Ee$ joining these two vertices.

      Similarly, if $\ga_{\mb{i}_l}^{j_0}$ is homotopic to $\ga_{\mb{i}_l,1}\cup\dots \cup \ga_{\mb{i}_l,k}$, and the curves $\ga_{\mb{i}_l,1},\dots,\ga_{\mb{i}_l,k}$ and $\ga_{\mb{i}_l}^{j_0-1}$ are identified with the vertices $v_{s_1},\dots,v_{s_k}$ and $v_{t_{j_0-1}}$, respectively. Then we define an edge $e_{t_{j_0-1},s_i}\in \Ee$ joining $v_{t_{j_0-1}}$ with each $v_{s_i}$, for $1\leq i\leq k$.
\end{enumerate}
As a consequence of the construction, the family of curves $\{\gamma_I\}_{I\in \Vv}$ satisfies the following property.
\begin{lemma}\label{lm_tree}
Let $\{\ga_I\}_{I\in \Vv(\Tt)}$ be a finite family of the curves constructed from above. Then the simplicial length of the approximation of each $\ga_I$ is bounded by $X(V,D)$, and in particular, we have the following.
  \begin{enumerate}
    \item If a vertex $v_j\in \Vv$ is a leaf, i.e., it has no children, then the corresponding curve $\ga_{v_j}$ is a point curve in $M$.
    \item If a vertex $v_j$ has a single child $v_s$, then the corresponding curve $\ga_{v_j}$ is homotopic to $\ga_{v_s}$ through a homotopy of width bounded by $D$.
    \item If a vertex $v_j$ has $k$ children $v_{s_1},\dots, v_{s_k}$, then the corresponding curve $\ga_{v_j}$ is homotopic to $\ga_{v_{s_1}}\cup\dots \cup \ga_{v_{s_k}}$ through a homotopy of width bouned by $D+W(V,D)$, where $W(V,D)$ is the same function as in Lemma~\ref{lm_brk}.
  \end{enumerate}
\end{lemma}

\item[(4)] Homotopy with marked points.

Given a closed curve $\ga$ and its associated family of curves $\{\ga_I\}_{I\in \Vv(\Tt)}$ as constructed above, we now develop a new homotopy based on this family which contracts $\ga$ to a marked point $p\in \ga$. This construction can be viewed as a generalization of Lemma~\ref{lm_wsep}. The construction will be done inductively as the following.

Without lost of generality, we assume a curve $\ga_i\in \{\ga_I\}_{I\in \Vv(\Tt)}$ corresponds to a vertex $v_i\in \Vv$. And we denote by $p_i$ a marked point in $\ga_i$. We first discuss in two cases according to the cases in Lemma~\ref{lm_tree}. The choice of the marked points on subsequent curves will also be based on this construction.

  \begin{enumerate}
  \item[(a),(b)] If $v_i$ has a single child $v_j$, let $\ga_i,\ga_j$ be the corresponding curves, and $p_i\in \ga_i$ is the marked point. We denote by $H(t,s)$ the homotopy between $\ga_i$ and $\ga_j$, where $H(t,0)=\ga_i(t)$ and $H(t,1)=\ga_j(t)$. By assumption, the width of this homotopy $w_H\leq D$.

      Let $\sg_i:[0,1]\ra M$ be the trajectory of the marked point $p_i$ under $H$. In other words, if $p_i=\ga(t_0)$, then $\sg_i(s)=H(t_0,s)$.  We define the marked point $p_j$ on $\ga_j$ to be the other endpoint of $\sg_i$, i.e., $p_j:=H(t_0,1)$.

      Then $\ga_i$ is homotopic to $\sg_i\cup \ga_j \cup -\sg_i$ through $\sg_i([0,t])\cup H([0,1],t)\cup -\sg_i([0,t])$, $0\leq t\leq 1$. The width of this homotopy is bounded by $\max\{2\length(\sg_i), w_H\}\leq 2D$.

  \item[(c)] Similarly, if $v_i$ has $k$ children $v_{s_1},\dots, v_{s_k}$. Let $H(s,t)$ be the homotopy between the corresponding curves $\ga_i$ and $\ga_{v_{s_1}}\cup\dots \cup \ga_{v_{s_k}}$. Let $p_i\in \ga_i$ be the marked point. By Lemma~\ref{lm_brk}, the curves $\{\ga_{v_{s_1}},\dots ,\ga_{v_{s_k}}\}$ has a common base point $p_s$. In this case, we choose the marked point $p_{s_j}=p_s$ on each $\ga_{s_j}$. However, $p_s$ may not be the destination of $p_i$ under $H$. Therefore, if $\ga_i(t_0)=p_i$ and $H(t_1,1)=p_s$, for some $t_0,t_1$, w.l.o.g, assume $t_1\geq t_0$,  we first apply a rotation $R(s,t)=H(t+s,0)$ of $\ga_i$, for $s\in[0,s_0]$, so that $R(t_0,0)=p_i$ and $R(t_0,s_0)=H(t_1,0)$. The width $w_R\leq \frac12 \length(\ga)\leq \frac32 D$.

      Let $\sg_i:[0,1+s_0]\ra M$ be the concatenation of the trajectory of $p_i$ under the rotation $R$ and homotopy $H$, where $\sg_i(0)=p_i$ and $\sg_i(1+s_0)=p_s$. Then $\ga_i$ is homotopic to $\sg_i\cup (\cup_{j=1}^k \ga_{s_j}) \cup -\sg_i$, through the curves $\sg([0,s_0+t])\cup  H([0,1],t)\cup -\sg_i([0,s_0+t])$. In this case the length of $\sg_i$ is bounded by $\frac32 D + D +W(V,D)$ and hence the width of this homotopy is bounded by $5D+2W(V,D)$.
    \end{enumerate}

Now, inductively, the root curve $\ga_0$ is homotopic to $\sg_1 \cup \ga_1 \cup -\sg_1$, then $\sg_1 \cup \sg_2 (\cup \ga_2 \dots) \cup -\sg_2 \cup -\sg_1$, $\dots$, and $\cup_{i\in \Vv} (\sg_i \cup -\sg_i)$. The width of this homotopy is bounded by $2D\x h_1+(5D+2W)\x h_2$, where $h_1,h_2$ are the number of steps (a)(b), (c) performed in the above construction, respectively. Note that $h_1+h_2=h(\Tt)$ which is the height of the tree $\Tt$. Finally, one may contract $\cup_{i\in \Vv} (\sg_i \cup -\sg_i)$ to the point $p$. Because each leaf $\sg_i \cup -\sg_i$ can be contracted at the same time, therefore, the width of this contraction is again bounded by $2D\x h_1+(5D+2W)\x h_2$, where $h_1+h_2=h(\Tt)$.

\item[(5)] Bounded height construction.

To complete the proof of our main theorem, we show that if $\ga$ is a closed curve of length $\leq 3D$ and $\{\ga_{v_I}\}_{v_I\in \Vv(\Tt)}$ is the associated family in the above construction, then there is a new family $\{\ga_{v_I}\}_{v_I\in \ti{\Vv}(\ti{\Tt})}$, which still satisfies the conclusion of Lemma~\ref{lm_tree} with the height of $\ti{\Tt}$ bounded in terms of $V,D$.

Indeed, for each $\ga_{v_I}$, where $v_I\in \Vv$, let us denote by $\Ll_I$ the subtree of $\Tt$ whose root is $v_I$. First note the following.

\begin{lemma}\label{lm_rep}
  Let $\ga_1$, $\ga_2$ be two closed curves of length $\leq 3D$. And let $\al_1$, $\al_2$ be their approximation as in Lemma~\ref{lm_simp}. If $\al_1=\al_2$ up to a reparametrization, then $\ga_1$ is homotopic to $\ga_2$ through a homotopy $H$ with $w_H\leq Y(V,D)$.
\end{lemma}

\begin{proof}
  The proof is the same as in Lemma~\ref{lm_simp}. Observe that the curve $ \ga([t_k,t_{k+1}])\cup -\sg_{i+1,k} \cup  -\ga_{kl}\cup \sg_{ik}$ can also be contracted to the point $\ga(t_k)$ with in $B_s$, and therefore inducing a homotopy between $\al_2$ and $\ga_2$ with width bounded by $G_\Sg$. Therefore, one may take $Y(V,D)=2G_\Sg (V,D)$. 
\end{proof}


Let us denote by $h(\Ll)$ the height of a subtree $\Ll\su \Tt$. Since $\Sg$ is a finite simplicial complex with $N$ vertices and at most $N^2$ edges, the number $N_0$ of closed curves, up to reparametrization, whose simplicial length is less than or equal to $X(V,D)$ is bounded by $(N^2+1)^X$.

Now if the height $h(\Tt)\geq N_0+1$, then there exists $v_1,v_2\in \Vv$ such that the approximation of $\ga_{v_1}$ and $\ga_{v_2}$ are the same, up to reparametrizaion. And $v_2\in \Vv (\Ll_1)$, i.e. $v_2$ is in the subtree with root $v_1$.

By Lemma~\ref{lm_rep}, $\ga_{v_1}$ and $\ga_{v_2}$ are homotopic with width $\leq Y(V,D)$. In this case, we construct a new tree $\Tt_1$ by replacing $\Ll_1$ with $\Ll_2$. Note that $\Tt_1$ is indeed a tree, i.e., contains no loops, since $v_2\in \Vv (\Ll_1)$.

Denote by $v_0$ the parent of $v_1$. Again, we discuss by cases as in Lemma~\ref{lm_tree}.
\begin{enumerate}
  \item[(a)(b)] If $v_1$ is the single child of $v_0$, then in $\Tt_1$, $v_2$ is the single child of $v_0$ and $\ga_{v_0}$ is homotopic to $\ga_{v_2}$ with width bounded by $D+Y(V,D)$.
  \item[(c)] If $v_0$ has $k$ children $v_{s_1},\dots, v_1, \dots, v_{s_{k-1}}$, then in $\Tt_1$, $\ga_{v_0}$ is homotopic to $\ga_{v_{s_1}}\cup\dots\cup  \ga_{v_2} \cup\dots \cup \ga_{v_{s_{k-1}}}$ through a homotopy of width bounded by $D+Y(V,D)+W(V,D)$.
\end{enumerate}

Observe that even though we reduced the height of a branch of $v_0$, in general, the height $h(\Tt_1)\leq h(\Tt)$, where the equality may happen. However, if the height $h(\Tt_1)\geq N_0+1$, then we may still apply the above construction to replace a subtree until $h(\Tt_k)\leq N_0\leq (N^2+1)^X$.

On the other hand, if in $\Tt_1$, there is some $v_3\in \Ll_{v_2}$ such that $\ga_{v_3}$ and $\ga_{v_2}$ have the same approximation, then in particular, $\ga_{v_1}$ and $\ga_{v_3}$ also have the same approximation. Hence, in this case $v_0$ can be connected to $v_3$ with homotopies of width bounded by $D+Y$ or $D+Y+W$, respectively for the cases (a)(b) and (c). 
\end{enumerate}
And therfore by (4) and (5), we have proved that the total width of contracting the curve $\ga$ to a point is bounded by 
$$B(V,D,N)\leq (N^2+1)^X \cdot (7D+2Y+W).$$
\end{proof}
We may now finish the proof of Theorem~\ref{thm_main}.
\begin{proof}[Proof of Theorem~\ref{thm_main}]
Let $M$ be a manifold satisfies the condition in Theorem~\ref{thm_main}. Suppose that there is no closed geodesic of length $\leq 3D$. Then Theorem~\ref{thmW} implies that every loop in $M$ may be contracted via a homotopy with width bounded by $B(V,D,N)$. This further implies by Theorem~\ref{thmD} that the depth $S_p(M,3D)\leq \max\{3D,2B+2D\}$.

As before, let $\OM_p(M)$ be the space of continuous maps $\{S^1\ra M\}$ based at $p\in M$ and $\OM^L_pM$ the subspace where every loop is of length $\leq L$. Now by taking the integer $k=2$ in Theorem~\ref{thm4_q}, we conclude that for every positive integer $m$, every map $f:S^{m}\ra \OM_pM$ is homotopic to a map $\ti{f}:S^m\ra \OM^F_pM$, where
$$F=F(m,V,D,N)=10\cdot m+D+(2m-1)\cdot S_p.$$
And in particular, the length of a closed periodic geodesic does not exceed $F(m,V,D,N)$.

Because our manifold is simply-connected, for $l\geq 2$, suppose it is $(l-1)-$connected but not $l-$connected , the above argument shows that there is a closed geodesic of length $\leq F(l,V,D)$. Note that $F$ is increasing in $l$. If $M$ is $(n-1)-$connected, then by Hurewicz theorem (see \cite[Theorem~4.32]{hatcher2002algebraic}), then $\pi_n(M)\cong H^n(M)\neq 0$, and hence we can take $A=F(n,V,D,N)$ in the above argument.
\end{proof}

\end{document}